\documentclass[preprint,12pt]{elsarticle}

\usepackage{amssymb, url}
\usepackage{lscape}
\usepackage{color}

\journal{}

\begin{document}

\begin{frontmatter}



\title{On the $\mathbb{F}_2$-linear relations of Mersenne Twister pseudorandom number generators}


\author[harase]{Shin Harase\corref{cor1}}
\ead{harase@craft.titech.ac.jp}
\address[harase]{Graduate School of Innovation Management, Tokyo Institute of Technology, 
W9-115, 2-12-1 Ookayama, Meguro-ku,Tokyo, 152-8550, Japan.}

\cortext[cor1]{Corresponding author}

\begin{abstract}

Sequence generators obtained by linear recursions over the two-element field $\mathbb{F}_2$\textcolor{red}{, 
i.e., $\mathbb{F}_2$-linear generators,} are widely used as pseudorandom number generators. 
For example, the Mersenne Twister MT19937 is one of the most successful applications. 
An advantage of such generators is that we can assess them \textcolor{blue}{quickly} by using theoretical criteria, 
such as the dimension of equidistribution with $v$-bit accuracy. 
To compute these dimensions, several \textcolor{red}{polynomial-time lattice reduction algorithms} have been proposed 
\textcolor{red}{in the case of $\mathbb{F}_2$-linear generators}. 

In this paper, \textcolor{red}{in order to assess non-random bit patterns 
in dimensions that are higher than the dimension of equidistribution with $v$-bit accuracy,}
we focus on 
the relationship between points in the Couture--L'Ecuyer dual lattices and 
$\mathbb{F}_2$-linear relations 
on the most significant $v$ bits of output sequences, and  
consider a new figure of merit $N_v$ based on the minimum weight of 
$\mathbb{F}_2$-linear relations whose degrees are minimal for $v$.
\textcolor{red}{Next, we numerically} show that 
MT19937 has low-weight $\mathbb{F}_2$-linear relations in dimensions higher than 623, 
and show that some output vectors with specific lags are rejected 
or have small $p$-values in birthday spacings tests. We also report that 
some variants of Mersenne Twister, such as WELL generators, 
are significantly improved from the perspective of $N_v$. 
\end{abstract}


\end{frontmatter}


\newtheorem{theorem}{Theorem}
\newtheorem{lemma}[theorem]{Lemma}
\newtheorem{corollary}[theorem]{Corollary}
\newdefinition{definition}{Definition}
\newdefinition{remark}{Remark}
\newproof{proof}{Proof}
\newtheorem{proposition}[theorem]{Proposition}

\section{Introduction}\label{sec:intro}

The {\it Mersenne Twister} MT19937 is a pseudorandom number generator 
developed by Matsumoto and Nishimura \cite{MT19937}. 
This generator has the following advantages: 
(i) Its generation speed is very fast; 
(ii) it has a large period of $2^{19937}-1$; 
(iii) it has high-dimensional equidistribution property (i.e., $623$-dimensionally equidistributed).
The algorithm of the Mersenne Twister is based on a linear recurrence relation over the two-element field $\mathbb{F}_2$. 
For such a class of generators (so-called {\it $\mathbb{F}_2$-linear generators}), the following two quality criteria 
are well-known: 
(i) the dimension of equidistribution with $v$-bit accuracy $k(v)$ for each $v$ ($1 \leq v \leq w$ and $w$ indicates the word size of machines) \textcolor{red}{is large enough} and 
(ii) the number $N_1$ of nonzero terms in a characteristic polynomial \textcolor{red}{is large enough (see Section~\ref{sec:quarity criteria} for details)}. 
From this perspective, \textcolor{blue}{MT19937 was 
much superior to all other classical pseudorandom number generators when it appeared, 
and it has since be the most widely used generator in Monte Carlo simulations. }

\textcolor{red}{However, it may not be sufficient to use only the above two criteria for assessing pseudorandom number generators. 
A motivation of this paper is to detect non-random bit patterns in dimensions 
that are higher than $k(v)$. 
In such dimensions, we always have certain bits of output whose sum becomes 0 over $\mathbb{F}_2$, 
in the case of $\mathbb{F}_2$-linear generators. 
Such relations are said to be {\it $\mathbb{F}_2$-linear relations}, but they are usually hidden. 
In fact, when there exist $\mathbb{F}_2$-linear relations whose numbers of terms are small (e.g, the number of terms $\leq 6$) 
in a low-dimensional projection, it is likely to observe some deficiencies for small sample sizes. 
As a previous work, Matsumoto and Nishimura \cite{MR1958868} gave theoretical justification for this fact in terms of coding theory.
Thus, we should avoid such $\mathbb{F}_2$-linear relations.}

\textcolor{red}{In this paper, in order to assess non-random bit patterns, 
we focus on the $\mathbb{F}_2$-linear relations whose degrees are minimal for the most significant $v$ bits given, 
which correspond to the shortest vectors in the Couture--L'Ecuyer dual lattice \cite{CL2000} for computing $k(v)$, 
and we develop a new figure of merit $N_v$ based on the minimum number of terms of the $\mathbb{F}_2$-linear relations.} 
The value $N_v$ can be considered as a quality \textcolor{red}{criterion} in dimensions higher than $k(v)$ 
and as a multi-dimensional generalization of $N_1$. 
\textcolor{blue}{Next,} we assess the Mersenne Twister MT19937 and its variants in terms of $N_v$'s, 
and show that $N_v$'s of MT19937 are small, 
relative to the WELL generators \cite{PannetonLM06}. 
We also report that MT19937 has some deviations in birthday spacings test 
\cite{Marsaglia1985, Knuth:1997:ACP:270146, MR1823110, MR2404400} 
for non-successive output values\textcolor{red}{,
in accordance with the existence of $\mathbb{F}_2$-linear relations. 
As far as the author knows, 
there has been no report of deviations of MT19937 
except for linear complexity tests and poor initialization until now (see \cite{MR2404400, MatsumotoSHN06, PannetonLM06}).}

The rest of this paper is organized as follows. 
In Section~\ref{sec:F2-linear generators}, we recall \textcolor{red}{the} framework of $\mathbb{F}_2$-linear generators. 
In Section~\ref{sec:quarity criteria}, we explain the terminologies of $k(v)$ and $N_1$. 
In Section~\ref{sec:lattice structure}, \textcolor{red}{ 
we briefly survey the Couture--\textcolor{blue}{L'Ecyuer} dual lattice method \cite{CL2000} for computing $k(v)$, 
which will be used in later sections.} 
In Section~\ref{sec:criterion}, we show the relationship between 
$\mathbb{F}_2$-linear relations and points in the Couture--L'Ecuyer dual lattices, 
define a new figure of merit $N_v$, 
and give an algorithm for computing $N_v$ using Gray codes. 
In Section~\ref{sec:linear relation for MT}, 
we analyze the Mersenne Twister MT19937 in terms of both $N_v$'s and $\mathbb{F}_2$-linear relations. 
In Section~\ref{sec:birthday spacings test}, we report some deviations of MT19937 
in the birthday spacings tests with selected lags. 
Section~\ref{sec:other generators} is devoted to the analysis of other $\mathbb{F}_2$-linear generators, 
such as the WELL generators. We also introduce a new tempering parameter of MT19937 
in order to optimize $k(v)$ as an improvement of the author's previous work \cite{Harase2009}. 
Our conclusions are presented in Section~\ref{sec:conclusions}.

\section{$\mathbb{F}_2$-linear generators} \label{sec:F2-linear generators}
Mersenne Twister generators belong to a general class of pseudorandom number generators based on the following
matrix recurrences \textcolor{red}{over the two-element field $\mathbb{F}_2 := \{ 0, 1 \}$}:
\begin{eqnarray}
\mathbf{x}_i & := & \mathbf{Ax}_{i-1}, \label{eqn:transition}\\
\mathbf{y}_i & := & \mathbf{Bx}_i, \label{eqn:transformation}\\
u_i & := & \sum_{l = 1}^w y_{i, l-1}2^{-l} = 0.y_{i,0}y_{i,1} \cdots y_{i, w-1}, \label{eqn:real} 
\end{eqnarray}
where $\mathbf{x}_i = {}^t(x_{i,0}, \ldots, x_{i, p-1}) \in \mathbb{F}_2^p$ is 
the $p$-bit {\it state vector} at step $i$ \textcolor{red}{(${}^t$ denotes the {\it transpose} of a vector)}; 
$\mathbf{y}_i = {}^t(y_{i, 0}, \ldots, y_{i, w-1}) \in \mathbb{F}^w$ is the $w$-bit {\it output vector} at step $i$; 
$p$ and $w$ are positive integers \textcolor{red}{($w$ indicates the word size of machines)}, 
$\mathbf{A}$ is a $p \times p$ {\it transition matrix} with elements in $\mathbb{F}_2$, 
and $\mathbf{B}$ is a $w \times p$ {\it output transformation matrix} with elements in $\mathbb{F}_2$. 
\textcolor{red}{ 
We identify a $w$-dimensional vector $\mathbf{y}_i$ and a binary expansion $u_i$  
with an unsigned $w$-bit binary integer and  a real number in the interval $[0,1)$, respectively. 
The output sequence $\{ u_i \}$ is supposed 
to imitate independent random variables that are uniformly distributed over $[0, 1)$.}
This framework is said to be the {\it $\mathbb{F}_2$-linear generator}. 
\textcolor{red}{We refer the reader to \cite{LP2009, MatsumotoSHN06} for details.}

Let $P(z) := \det(\mathbf{I}z - \mathbf{A})$ be the {\it characteristic polynomial} of $\mathbf{A}$. 
The recurrence (\ref{eqn:transition}) has the period length $2^p-1$ 
(its maximal possible value) if and only if $P(z)$ is a primitive polynomial modulo $2$ 
(see \cite{Niederreiter:book, Knuth:1997:ACP:270146}). 
When this maximum is reached, 
we say that the $\mathbb{F}_2$-linear generator has the {\it maximal period}. 
For simplicity, we assume the maximal-period condition throughout this paper. 

From Section~2.3.5 in \cite{PannetonPhd}, 
the original Mersenne Twister \cite{MT19937} fits the above framework by the \textcolor{red}{$((n_1w-r) \times (n_1w-r))$-}transition matrix
\begin{eqnarray} \label{eqn:MT transition matrix}
\mathbf{A} = \left(
        \begin{array}{ccccccc}
         {} & {} & {} & \mathbf{I}_{w} & {} & {} & \mathbf{S}\\
         \mathbf{I}_{w} & {} & {} & {} & {} & {} & {}\\
         {} & \mathbf{I}_{w} & {} & {} & {} & {} & {}\\  
         {} &  {} & \ddots & {} & {} & {} & {}\\
         {} &  {} & {} & {} & \ddots & {} & {}\\
         {} &  {} & {} & {} & {} & \mathbf{I}_{w - r} & {}\\
         \end{array}
      \right), \ 
\mathbf{S} = \tilde{\mathbf{A}} \left(
      \begin{array}{cc}
      \mathbf{0} & \mathbf{I}_{w-r}\\
      \mathbf{I}_{r} & \mathbf{0}
               \end{array}
      \right),
\end{eqnarray}
where \textcolor{red}{$r$ is an integer with $0 \leq r \leq w-1$},
$\mathbf{I}_w, \mathbf{I}_{r}, \mathbf{I}_{w-r}$ are the identity matrices of size $w, r, w-r$, respectively, 
\textcolor{red}{$\tilde{\mathbf{A}}$ is a suitable $w \times w$ matrix with elements in $\mathbb{F}_2$, 
$\mathbf{S}$ and $\mathbf{I}_{w}$ in the first block of row are located at the $n_1$-th and $(n_1-n_2)$-th blocks of column, respectively, 
and $p = n_1w-r$}. (Note that MT19937 has \textcolor{red}{$(w, n_1, n_2, r) = (32, 624, 397, 31)$ and $p = 19937$, 
which is a Mersenne exponent.}) 
The matrix $\mathbf{B}$ is the representation matrix of the following transformation (so-called the {\it tempering} \cite{MK1994}): 
\begin{eqnarray}
\mathbf{z} & \leftarrow & {\rm trunc}_w (\mathbf{x}_i) \label{eqn:MT output transformation1} \\
\mathbf{y}_i & \leftarrow & \mathbf{Tz} \label{eqn:MT output transformation2}
\end{eqnarray} 
where ``$\leftarrow$'' represents the assignment statement, 
${\rm trunc}_w(\mathbf{x}_i)$ denotes the vector ${}^t(x_{i,0}, \ldots, x_{i, w-1})$, 
which is formed by leading $w$ coordinates of $\mathbf{x}_i$, and 
\textcolor{red}{$\mathbf{T}$ is a suitable $w \times w$ invertible matrix, i.e., 
$\mathbf{B} := (\mathbf{T} , \mathbf{0} , \ldots , \mathbf{0})$.}

\section{Quality criteria} \label{sec:quarity criteria}

Following \cite{LP2009}, we \textcolor{blue}{briefly} recall figures of merit for $\mathbb{F}_2$-linear generators. 
A primary requirement is that $k$-dimensional vectors $(u_i, u_{i+1}, \ldots, u_{i+k-1})$ are uniformly distributed over the unit hypercube $[0, 1)^k$ for large $k$. 
\textcolor{blue}{For this,} we often use the terminology of the {\it dimension of equidistribution with $v$-bit accuracy}. 
Let $\Psi_k$ be the multiset of $k$-dimensional vectors, from all possible initial states $\mathbf{x}_0$:
\begin{eqnarray*}
\Psi_k := \{ (u_0, \ldots, u_{k-1}) \ | \ \mathbf{x}_0 \in \mathbb{F}_2^p \} \subset [0, 1)^k. \label{eqn:multiset}
\end{eqnarray*}
Let us \textcolor{red}{divide} axis $[0,1)$ into $2^v$ pieces.
Then, $[0, 1)^k$ is divided into $2^{kv}$ cubic cells of equal size.
The generator is said to be {\it $k$-dimensionally equidistributed with $v$-bit accuracy} 
if each cell contains exactly the same number of points of $\Psi_k$, i.e., $2^{p-kv}$ points. 
The largest value of $k$ with this property is called 
the {\it dimension of equidistribution with $v$-bit accuracy}, denoted by $k(v)$.
As a criterion of uniformity, larger $k(v)$ for each $1 \leq v \leq w$ is desirable (see \cite{Tootill}). 
We have a trivial upper bound $k(v) \leq \lfloor p/v \rfloor$. 
The gap $d(v) := \lfloor p/v \rfloor -k(v)$ is called the {\it dimension defect} at $v$, and their sum 
$\Delta:=\sum_{v = 1}^{w} (\lfloor p/v \rfloor - k(v))$
is called the {\it total dimension defect}. 
If $\Delta = 0$, the generator is said to be {\it maximally equidistributed}. 
Note that MT19937 has $\Delta = 6750$. 

As a secondary requirement, we may
consider whether the number $N_1$ of nonzero coefficients 
for $P(z)$ is large enough or not 
(see \cite{springerlink:10.1007/BF01029989, MR1155576}). 
For example, \textcolor{red}{generators for which $P(z)$ is a trinomial or a pentanomial fail statistical tests 
\cite{Lindholm1968, Matsumoto:1996:SDR:232807.232815, MR1958868}, 
so that they should be avoided. 
MT19937 has $N_1 = 135$, \textcolor{blue}{so that} it  has long-lasting  impact from poor initialization, e.g., $\mathbf{x}_0$ that contains only a few bits set to 1
(see \cite{PannetonLM06}). 
WELL generators \cite{PannetonLM06} have $N_1 \approx p/2$, and overcome the above drawback. 
We \textcolor{blue}{note} that such phenomena of initialization will not happen practically if we take care of initialization routines.}


\section{Lattice structures} \label{sec:lattice structure}
We briefly recall a lattice method for computing 
$k(v)$ in terms of the Couture--L'Ecuyer dual lattices \textcolor{red}{by following \cite{CL2000, LP2009}}. 
\textcolor{blue}{Recently, the faster algorithms \cite{Shin2011141, HMS2010} using the original lattices \cite{CLT93a, Tezuka94} (not the dual lattices) were proposed. 
The aim of this paper is to extract other information from the dual lattices, so that we \textcolor{blue}{revisit} the Couture--L'Ecuyer dual lattices.}

Let $K$ denote the formal power series field $K := \mathbb{F}_2((z^{-1})) = 
\{ \sum_{i = i_0}^{\infty} a_i z^{-i}\ | \ a_i \in {\mathbb{F}}_2, i_0 \in \mathbb{Z} \}$. 
For $a(z) = \sum_{i = i_0}^{\infty} a_i z^{-i} \in K$, 
\textcolor{red}{
we put a standard norm by $|a(z)| := \max \{ -i \in \mathbb{Z} \ | \ a_i \neq 0 \}$ if $a(z) \neq 0$ and 
$|a(z)| :=-\infty$ if $a(z) = 0$.
For a vector $\mathbf{a}(z) = {}^t(a_0(z), a_1(z), \ldots, a_{v-1}(z)) \in K^v$,
we define its {\it norm} (or its {\it length}) by $||\mathbf{a}(z)|| := {\rm max}_{1 \leq l \leq v}|a_{l-1}(z)|$}. 

A subset $L \subset K^v$ is said to be 
an $\mathbb{F}_2[z]$-{\em lattice} if there exists a $K$-linear basis
$\{ \mathbf{v}_1(z), \mathbf{v}_2(z), \ldots, \mathbf{v}_v(z) \}$ of $K^v$ such that
$L$ is their span over $\mathbb{F}_2[t]$, i.e., \textcolor{red}{$L=\langle \mathbf{v}_1(z), \mathbf{v}_2(z), \ldots, \mathbf{v}_v(z)
\rangle_{\mathbb{F}_2[t]}$}.
Such a set of vectors is called a {\it basis} of $L$. 
\textcolor{red}{A reduced basis $\{ \mathbf{v}_1(z), \ldots, \mathbf{v}_v(z) \}$ is defined as follows: 
(i) $\mathbf{v}_1(z)$ is a nonzero shortest vector in $L$, and  (ii) for $l = 2, \ldots, v$, $\mathbf{v}_l(z)$ is 
a shortest vector among the set of vectors $\mathbf{v}(z)$ in $L$ 
such that $\mathbf{v}_1(z), \ldots, \mathbf{v}_{l-1}(z), \mathbf{v}(z)$ are 
linearly independent over $K$. 
It is not unique, but the numbers $\nu_l := ||\mathbf{v}_l(z)||$ ($l = 1, \ldots, v$) are uniquely determined 
by a given lattice $L$, 
and $\nu_1, \ldots, \nu_v$ are called the {\em successive minima} of $L$ (see \cite{Mahler1941}).}

We consider an $\mathbb{F}_2$-linear generator. 
For a given nonzero initial state $\mathbf{x}_0 \in \mathbb{F}_2^p$, 
we define the following formal power series $G_{l-1}(z)$ of $l$th bits of the integer output (i.e., $y_{0, l-1}, y_{1, l-1}, 
y_{2,l-1}, \cdots$): 
\[ G_{l-1}(z) := \sum_{i = 0}^{\infty} y_{i, l-1} z^{-i-1} = y_{0, l-1}z^{-1} + y_{1, l-1}z^{-2} + y_{2, l-1}z^{-3} + 
\cdots \in \mathbb{F}_2((z^{-1})). \]
Note that $G_{l-1}(z)$ has a rational form $G_{l-1} (z) = h_{l-1} (t)/P(z)$, where $h_{l-1}(z) \in \mathbb{F}_2[z]$ and $\deg h_{l-1} (z) < \deg P(z)$. If $P(z)$ is irreducible, 
let $h_0^{-1}(z)$ be a polynomial that is a multiplicative inverse to $h_0(z)$ modulo $P(z)$. 
We set ${\bar h}_{l-1}(z) := h_0^{-1}(z) h_{l-1}(z) \mbox{ mod } P(z)$ ($2 \leq l \leq v$). 
We consider the following vectors: \textcolor{red}{
$\mathbf{w}_1(z)  :=  {}^t(P(z), 0, 0, \ldots, 0),
\mathbf{w}_2(z)  :=  {}^t(-{\bar h}_1(z), 1, 0, \ldots, 0),
\mathbf{w}_3(z)  :=  {}^t(-{\bar h}_2(z), 0, 1, \ldots, 0),
 \ldots, \mathbf{w}_v(z)  :=  {}^t(-{\bar h}_{v-1}(z), 0, 0, \ldots, 1)$},
and construct an $\mathbb{F}_2[z]$-lattice $\mathcal{L}_v^* := \langle \mathbf{w}_1(z), \ldots, \mathbf{w}_v(z) \rangle_{\mathbb{F}_2[z]} \subset \mathbb{F}_2^v[z]$, 
which is said to be the Couture--L'Ecuyer {\it dual lattice} \cite{CL2000}. 

\begin{theorem}[\cite{CL2000}]\label{th:CL2000}
\textcolor{blue}{C}onsider an $\mathbb{F}_2$-linear generator started from a nonzero initial state vector.
Assume that the characteristic polynomial $P(z)$ of $\mathbf{A}$ is primitive. 
Then, $k(v)= \nu_1^{*}$, 
where $\nu_1^{*}$ is the first successive minimum of $\mathcal{L}_v^*$.
\end{theorem}
We can obtain a reduced basis by using some polynomial-time lattice basis reduction algorithms 
(e.g., \cite{MR736459, Lenstra1985,MS2003,Schmidt1991181}).

\section{A new figure of merit for $\mathbb{F}_2$-linear generators} \label{sec:criterion}

Usually, when we assess an $\mathbb{F}_2$-linear generator, 
we only see the length of a shortest vector (i.e., $k(v)$) as the first filter, 
and abandon the other information. In this section, 
we focus on the polynomial elements of vectors in $\mathcal{L}_v^*$,  
and develop a new figure of merit $N_v$ as 
a quality criterion in dimensions that are higher than $k(v)$ and 
as a multi-dimensional generalization of $N_1$. 

First, we arrange the relationship between  
$\mathbb{F}_2$-linear relations appeared on the most significant $v$ bits
and points in $\mathcal{L}_v^*$. 
\begin{proposition} \label{proposition1}
There exists an $\mathbb{F}_2$-linear relation
\begin{equation} \label{eqn:linear_relation}
\sum_{l = 1}^{v} \sum_{j = 0}^{k-1} w_{j, l-1} y_{{i + j}, l-1} = 0 \mbox{ for all } i \geq 0,
\end{equation}
if and only if ${}^t(w_0(z)), \ldots, w_{v-1}(z)) \in \mathcal{L}_v^*$, 
where $w_{l-1}(z) := \sum_{j = 0}^{k-1} w_{j, l-1} z^j \in \mathbb{F}_2[z]$.
\end{proposition}
\begin{proof}
We consider a linear combination
\begin{eqnarray} \label{eqn:linear combination}
G_0(z)w_0(z) + \cdots + G_{l-1}(z)w_{l-1}(z).
\end{eqnarray} 
For $i \geq 0$, the coefficient of $z^{-i-1}$ in (\ref{eqn:linear combination})
is $\sum_{l = 1}^{v} \sum_{j = 0}^{k-1} w_{j, l-1} y_{{i + j}, l-1}$, 
so that the coefficients of the negative power are all zero if and only if (\ref{eqn:linear_relation}) holds. 
Then, (\ref{eqn:linear combination}) is a polynomial.
On the other hand, (\ref{eqn:linear combination}) is also described as 
$( h_0(z)w_0(z) + \cdots + h_{v-1}(z)w_{v-1}(z))/P(z)$. 
Thus, (\ref{eqn:linear_relation}) is equivalent to
\begin{eqnarray} \label{eqn:Diophantine}
h_0(z) w_0(z) + \cdots + h_{v-1}(z) w_{v-1}(z) \equiv 0 \mbox{ mod } P(z).
\end{eqnarray}

Here, we assume (\ref{eqn:linear_relation}). 
By multiplying (\ref{eqn:Diophantine}) by $h_0^{-1}(z)$, 
we have $w_0(z) \equiv -  {\bar h_1}(z) w_1(z)- \cdots -  {\bar h_{v-1}}(z) w_{v-1}(z)\mbox{ mod } P(z)$.
In (\ref{eqn:Diophantine}), 
each of polynomial solutions ${}^t(w_0(z), \ldots, w_{v-1}(z))$ is written as $-a(z)\mathbf{w}_1(z) + h_1(z)\mathbf{w}_2(z) + \cdots + h_{v-1}(z) \mathbf{w}_v(z) = {}^t(-a(z)P(z)-{\bar h_1}(z)w_1(z) - \cdots - {\bar h_{v-1}}(z)w_{v-1}(z), 
w_1(z), \ldots, w_{v-1}(z))$ for a suitable $a(z) \in \mathbb{F}_2[z]$. 
Hence, ${}^t(w_0(z), \ldots, w_{v-1}(z)) \in \mathcal{L}_v^*$. 

Conversely, it is easy to see that 
all of $\mathbb{F}_2[z]$-linear combinations of $\mathbf{w}_1(z), \ldots, \mathbf{w}_v(z)$ satisfy 
(\ref{eqn:Diophantine}), 
because $\mathbf{w}_1(z), \ldots, \mathbf{w}_{v-1}(z)$ are solutions in (\ref{eqn:Diophantine}), respectively.
Thus, the proposition follows.
\end{proof}

Using the above proposition, from vectors of $\mathcal{L}_v^*$, 
we obtain information on $\mathbb{F}_2$-linear relations
in dimensions that are higher than $k(v)$. 
In particular, a nonzero shortest vector of 
$\mathcal{L}_v^*$ corresponds to  
a non-trivial $\mathbb{F}_2$-linear relation 
\begin{eqnarray} \label{eqn:minimal}
\sum_{l = 1}^{v} \sum_{j = 0}^{k(v)} w_{j,l-1} y_{i + j, l-1} = 0 \mbox{ for all } i \geq 0, 
\end{eqnarray}
whose degree is minimal for the most significant $v$ bits given. 
We call (\ref{eqn:minimal}) a {\it minimal $\mathbb{F}_2$-linear relation with $v$-bit accuracy}. 
In general, such a minimal $\mathbb{F}_2$-linear relation is not unique, 
because a shortest vector is not unique. 
Furthermore, we have no non-trivial $\mathbb{F}_2$-linear relation 
$\sum_{l = 1}^{v} \sum_{j = 0}^{k} w_{j,l-1} y_{i + j, l-1} = 0$ for $k < k(v)$. 
All of the minimal $\mathbb{F}_2$-linear relations with $v$-bit accuracy are included 
in all the vectors whose dimensions are higher than $k(v)$, 
so that the $(k(v)+1)$-dimensional case appears to be the most important. 

Here, to assess the $\mathbb{F}_2$-linear generators, 
let us consider whether or not the minimal $\mathbb{F}_2$-linear relations have simple regularity. 
The simplest way of checking this is to enumerate the number of nonzero coefficients $w_{j, l-1}$ in (\ref{eqn:minimal}). We call this the {\it weight}. 
When there exist low-weight $\mathbb{F}_2$-linear relations, 
the generator may have risks in some situations (see Remark~\ref{remark:weight discrepancy test}). 
Therefore, we define the {\it minimum weight} $N_v$ by the lowest weight 
for all the minimal $\mathbb{F}_2$-linear relations with $v$-bit accuracy in (\ref{eqn:minimal}), 
and propose $N_v$ as a new figure of merit for $\mathbb{F}_2$-linear generators.
When $v = 1$, the minimal $\mathbb{F}_2$-linear relation (\ref{eqn:minimal}) 
coincides with a characteristic polynomial $P(z)$, so that 
$N_v$ equals the number $N_1$ of nonzero coefficients of $P(z)$. 
Hence, we can interpret $N_v$ as a multi-dimensional generalization of $N_1$.  

For practical use, we give an algorithm for computing $N_v$ as follows.  
Let $\{ \tilde{{\bf w}}_1(z), \ldots, \tilde{{\bf w}}_v(z) \}$ be a reduced basis of $\mathcal{L}_v^*$.
From the uniqueness of the successive minima, 
we have an integer $v' \in \{ 1, \ldots, v \}$ such that   
$|| \tilde{{\bf w}}_1(z) || = \cdots = || \tilde{{\bf w}}_{v'}(z) || < || \tilde{{\bf w}}_{v'+1}(z) || \leq \cdots \leq || \tilde{{\bf w}}_v(z) ||$. 
Then, all the shortest vectors are described by
\begin{eqnarray} \label{eqn:Gray}
 \{ c_1 \tilde{{\bf w}}_1(z) + \ldots + c_{v'} \tilde{{\bf w}}_{v'}(z) \ | \ {}^t(c_1, \ldots, c_{v'}) \in \mathbb{F}_2^{v'} \setminus {}^t(0, \ldots, 0) \},
\end{eqnarray}
and they correspond to all the minimal $\mathbb{F}_2$-linear relations with $v$-bit accuracy. 
The number of the shortest vectors is $2^{v'} - 1$. 
Here, if we give coefficients $c_1, \ldots, c_{v'}$ by $v'$-bit Gray code order, 
it is possible to obtain another shortest vector by executing the addition only once. 
In addition, if $v'$ is small and $p$ is not too large (e.g., $v' \leq 32$ and $p \leq 19937$), 
we can compute the minimum weight $N_v$ within a practical time period. 

\begin{remark} \label{remark:weight discrepancy test}
\textcolor{red}{As theoretical justification,} we mention a strong relationship between our figure of merit $N_v$ 
and the {\it weight discrepancy test} proposed by Matsumoto and Nishimura \cite{MR1958868}.  
For simplicity, consider $k$ successive output values with $v$-bit accuracy, 
where $k > k(v)$, and 
let $\Phi$ be the map from the state vectors to $m := v \times k$ bits in the outputs: 
\begin{eqnarray} \label{eqn:code}
 \Phi: \mathbb{F}_2^{p} \to \mathbb{F}_2^{m}, \quad \mathbf{x}_0 \mapsto ({\rm trunc}_v(\mathbf{y}_0), 
{\rm trunc}_v(\mathbf{y}_{1}), \ldots, {\rm trunc}_v(\mathbf{y}_{k-1})).
\end{eqnarray}
The map $\Phi$ is $\mathbb{F}_2$-linear, so that the image $C \subset \mathbb{F}_2^m$ is a linear subspace. 
In coding theory, $C$ is said to be a {\it linear code}. 
The {\it dual code} $C^{\perp}$ of $C$ is defined by
\[ C^{\perp} := \{ \mathbf{c}' \in \mathbb{F}_2^m \ | \ \langle \mathbf{c}', \mathbf{c} \rangle = 0 \mbox{ for all } \mathbf{c} \in C \}, \]
where $\langle \mathbf{c}', \mathbf{c} \rangle = \sum_{i = 1}^m c_i' c_i$ is an inner product
for $\mathbf{c}' = {}^t (c_1', \ldots, c_m') \in \mathbb{F}_2^m$ and $\mathbf{c} = {}^t (c_1, \ldots, c_m) \in \mathbb{F}_2^m$. 
Note that $C^{\perp}$ contains the set of $\mathbb{F}_2$-linear relations on $m = k \times v$ bits (see \cite{MR1958868}). 

The weight discrepancy test is a theoretical test for estimating 
the deviation of the number of $1$'s on $m = v \times \textcolor{blue}{k}$ bits in (\ref{eqn:code}) from the binomial distribution. 
Matsumoto and Nishimura \cite{MR1958868} gave a formula for computing a risky sample size 
from a weight enumerator polynomial of $C$, 
which is computed via a weight enumerator polynomial of $C^{\perp}$ with 
$k$ being slightly greater than $k(v)$,
and via inversion by the MacWilliams identity. 
Their paper implies that if the minimum weight of vectors of $C^{\perp}$ is more than $15$ or $20$, 
a given generator is safe, but if the weights are too low, e.g., $\leq 6$, \textcolor{red}{(and especially if such $\mathbb{F}_2$-linear relations 
concentrate on a low-dimensional projection)}, 
there \textcolor{red}{is} a possibility of detecting deviations. 
In fact, we can identify $C^{\perp}$ with a set of all the vectors whose lengths are shorter than $k$, 
so that $N_v$ coincides with the minimum weight of vectors in $C^{\perp}$ in the case where $k = k(v) + 1$. 
Thus, when $N_v$ is small, a given generator may have some risks. 

A drawback of $N_v$ is that 
we have to execute exhaustive searches \textcolor{blue}{in} (\ref{eqn:Gray}) 
because the weight enumeration 
(or finding the minimum weight $N_v$) is NP-hard \cite{MR1481035}. 
However, from the viewpoint of speed and memory efficiency, 
the use of the Couture--L'Ecuyer dual lattice method appears to be much superior to 
the use of the Gaussian elimination on a $p \times m$ matrix in  \cite{MR1958868} 
when we construct a basis of $C^{\perp}$ (as an $\mathbb{F}_2$-linear vector space) for a large $p$. 
\end{remark}

\section{$\mathbb{F}_2$-linear relations of Mersenne Twister MT19937} \label{sec:linear relation for MT}
In this section, we numerically analyze $32$-bit 
Mersenne Twister MT19937 (i.e., $w = 32$) in terms of the method in Section~\ref{sec:criterion}. 
Tables~\ref{table:successive_minima1} and ~\ref{table:successive_minima2} list the successive minima $\nu_1^{*}, \nu_2^{*}, \ldots, \nu_v^{*}$ of $\mathcal{L}_v^{*}$,
 the dimension defect $d(v)$ at $v$, 
and our new figure of merit $N_v$ for each $1 \leq v \leq 32$. 
From Theorem~\ref{th:CL2000}, note that $\nu_1^{*} = k(v)$. 
As a result, $N_v$'s for lower bits are small. 

\begin{table}[h]
\caption{The successive minima, $d(v)$, and $N_v$ of MT19937.}
{\tiny
\begin{tabular}{|r||r|r|r|r|r|r|r|r|r|r|r|r|r|r|r|r|} \hline 
{} & $\mathcal{L}_{1}^{*}$ & $\mathcal{L}_{2}^{*}$ & $\mathcal{L}_{3}^{*}$ & $\mathcal{L}_{4}^{*}$ & $\mathcal{L}_{5}^{*}$ & $\mathcal{L}_{6}^{*}$ & $\mathcal{L}_{7}^{*}$ & $\mathcal{L}_{8}^{*}$ & $\mathcal{L}_{9}^{*}$ & $\mathcal{L}_{10}^{*}$ & $\mathcal{L}_{11}^{*}$ & $\mathcal{L}_{12}^{*}$ & $\mathcal{L}_{13}^{*}$ & $\mathcal{L}_{14}^{*}$ & $\mathcal{L}_{15}^{*}$ & $\mathcal{L}_{16}^{*}$  \\ \hline 
$\nu_{1}^{*}$ & $19937$ & $9968$ & $6240$ & $4984$ & $3738$ & $3115$ & $2493$ & $2492$ & $1869$ & $1869$ & $1248$ & $1246$ & $1246$ & $1246$ & $1246$ & $1246$ \\ 
$\nu_{2}^{*}$ & {} & $9969$ & $6848$ & $4984$ & $3738$ & $3115$ & $2493$ & $2492$ & $1869$ & $1869$ & $1868$ & $1246$ & $1246$ & $1246$ & $1246$ & $1246$ \\ 
$\nu_{3}^{*}$ & {} & {} & $6849$ & $4984$ & $3738$ & $3115$ & $2493$ & $2492$ & $1870$ & $1869$ & $1869$ & $1247$ & $1246$ & $1246$ & $1246$ & $1246$ \\ 
$\nu_{4}^{*}$ & {} & {} & {} & $4985$ & $3739$ & $3116$ & $3114$ & $2492$ & $1870$ & $1869$ & $1869$ & $1247$ & $1246$ & $1246$ & $1246$ & $1246$ \\ 
$\nu_{5}^{*}$ & {} & {} & {} & {} & $4984$ & $3738$ & $3114$ & $2492$ & $2491$ & $1869$ & $1869$ & $1868$ & $1246$ & $1246$ & $1246$ & $1246$ \\ 
$\nu_{6}^{*}$ & {} & {} & {} & {} & {} & $3738$ & $3115$ & $2492$ & $2492$ & $1869$ & $1869$ & $1869$ & $1247$ & $1246$ & $1246$ & $1246$ \\ 
$\nu_{7}^{*}$ & {} & {} & {} & {} & {} & {} & $3115$ & $2492$ & $2492$ & $1870$ & $1869$ & $1869$ & $1247$ & $1246$ & $1246$ & $1246$ \\ 
$\nu_{8}^{*}$ & {} & {} & {} & {} & {} & {} & {} & $2493$ & $2492$ & $1870$ & $1869$ & $1869$ & $1868$ & $1246$ & $1246$ & $1246$ \\ 
$\nu_{9}^{*}$ & {} & {} & {} & {} & {} & {} & {} & {} & $2492$ & $2491$ & $1869$ & $1869$ & $1869$ & $1247$ & $1246$ & $1246$ \\ 
$\nu_{10}^{*}$ & {} & {} & {} & {} & {} & {} & {} & {} & {} & $2492$ & $1869$ & $1869$ & $1869$ & $1247$ & $1246$ & $1246$ \\ 
$\nu_{11}^{*}$ & {} & {} & {} & {} & {} & {} & {} & {} & {} & {} & $1869$ & $1869$ & $1869$ & $1868$ & $1246$ & $1246$ \\ 
$\nu_{12}^{*}$ & {} & {} & {} & {} & {} & {} & {} & {} & {} & {} & {} & $1869$ & $1869$ & $1869$ & $1247$ & $1246$ \\ 
$\nu_{13}^{*}$ & {} & {} & {} & {} & {} & {} & {} & {} & {} & {} & {} & {} & $1869$ & $1869$ & $1247$ & $1246$ \\ 
$\nu_{14}^{*}$ & {} & {} & {} & {} & {} & {} & {} & {} & {} & {} & {} & {} & {} & $1869$ & $1868$ & $1246$ \\ 
$\nu_{15}^{*}$ & {} &{} & {} & {} & {} & {} & {} & {} & {} & {} & {} & {} & {} & {} & $1869$ & $1246$ \\ 
$\nu_{16}^{*}$ & {} & {} & {} & {} & {} & {} & {} & {} & {} & {} & {} & {} & {} & {} & {} & $1247$ \\ \hline
$d(v)$ &  $0$ & $0$ & $405$ & $0$ & $249$ & $207$ & $355$ & $0$ & $346$ & $124$ & $564$ & $415$ & $287$ & $178$ & $83$ & $0$ \\ \hline
$N_v$ &  $135$ & $10020$ & $393$ & $128$ & $44$ & $57$ & $38$ & $15$ & $10$ & $10$ & $40$ & $5$ & $5$ & $5$ & $5$ & $5$ \\ \hline
\end{tabular}
}
\label{table:successive_minima1}
\end{table}

\begin{table}[h]
\caption{The successive minima, $d(v)$, and $N_v$ of MT19937 (continued).}
{\tiny
\begin{tabular}{|r||r|r|r|r|r|r|r|r|r|r|r|r|r|r|r|r|} \hline 
{} & $\mathcal{L}_{17}^{*}$ & $\mathcal{L}_{18}^{*}$ & $\mathcal{L}_{19}^{*}$ & $\mathcal{L}_{20}^{*}$ & $\mathcal{L}_{21}^{*}$ & $\mathcal{L}_{22}^{*}$ & $\mathcal{L}_{23}^{*}$ & $\mathcal{L}_{24}^{*}$ & $\mathcal{L}_{25}^{*}$ & $\mathcal{L}_{26}^{*}$ & $\mathcal{L}_{27}^{*}$ & $\mathcal{L}_{28}^{*}$ & $\mathcal{L}_{29}^{*}$ & $\mathcal{L}_{30}^{*}$ & $\mathcal{L}_{31}^{*}$ & $\mathcal{L}_{32}^{*}$  \\ \hline 
$\nu_{1}^{*}$ & $623$ & $623$ & $623$ & $623$ & $623$ & $623$ & $623$ & $623$ & $623$ & $623$ & $623$ & $623$ & $623$ & $623$ & $623$ & $623$ \\ 
$\nu_{2}^{*}$ & $623$ & $623$ & $623$ & $623$ & $623$ & $623$ & $623$ & $623$ & $623$ & $623$ & $623$ & $623$ & $623$ & $623$ & $623$ & $623$ \\ 
$\nu_{3}^{*}$ & $1246$ & $623$ & $623$ & $623$ & $623$ & $623$ & $623$ & $623$ & $623$ & $623$ & $623$ & $623$ & $623$ & $623$ & $623$ & $623$ \\ 
$\nu_{4}^{*}$ & $1246$ & $623$ & $623$ & $623$ & $623$ & $623$ & $623$ & $623$ & $623$ & $623$ & $623$ & $623$ & $623$ & $623$ & $623$ & $623$ \\ 
$\nu_{5}^{*}$ & $1246$ & $1246$ & $623$ & $623$ & $623$ & $623$ & $623$ & $623$ & $623$ & $623$ & $623$ & $623$ & $623$ & $623$ & $623$ & $623$ \\ 
$\nu_{6}^{*}$ & $1246$ & $1246$ & $624$ & $623$ & $623$ & $623$ & $623$ & $623$ & $623$ & $623$ & $623$ & $623$ & $623$ & $623$ & $623$ & $623$ \\ 
$\nu_{7}^{*}$ & $1246$ & $1246$ & $1246$ & $623$ & $623$ & $623$ & $623$ & $623$ & $623$ & $623$ & $623$ & $623$ & $623$ & $623$ & $623$ & $623$ \\ 
$\nu_{8}^{*}$ & $1246$ & $1246$ & $1246$ & $624$ & $623$ & $623$ & $623$ & $623$ & $623$ & $623$ & $623$ & $623$ & $623$ & $623$ & $623$ & $623$ \\ 
$\nu_{9}^{*}$ & $1246$ & $1246$ & $1246$ & $1246$ & $623$ & $623$ & $623$ & $623$ & $623$ & $623$ & $623$ & $623$ & $623$ & $623$ & $623$ & $623$ \\ 
$\nu_{10}^{*}$ & $1246$ & $1246$ & $1246$ & $1246$ & $624$ & $623$ & $623$ & $623$ & $623$ & $623$ & $623$ & $623$ & $623$ & $623$ & $623$ & $623$ \\ 
$\nu_{11}^{*}$ & $1246$ & $1246$ & $1246$ & $1246$ & $1246$ & $623$ & $623$ & $623$ & $623$ & $623$ & $623$ & $623$ & $623$ & $623$ & $623$ & $623$ \\ 
$\nu_{12}^{*}$ & $1246$ & $1246$ & $1246$ & $1246$ & $1246$ & $624$ & $623$ & $623$ & $623$ & $623$ & $623$ & $623$ & $623$ & $623$ & $623$ & $623$ \\ 
$\nu_{13}^{*}$ & $1246$ & $1246$ & $1246$ & $1246$ & $1246$ & $1246$ & $623$ & $623$ & $623$ & $623$ & $623$ & $623$ & $623$ & $623$ & $623$ & $623$ \\ 
$\nu_{14}^{*}$ & $1246$ & $1246$ & $1246$ & $1246$ & $1246$ & $1246$ & $624$ & $623$ & $623$ & $623$ & $623$ & $623$ & $623$ & $623$ & $623$ & $623$ \\ 
$\nu_{15}^{*}$ & $1246$ & $1246$ & $1246$ & $1246$ & $1246$ & $1246$ & $1246$ & $623$ & $623$ & $623$ & $623$ & $623$ & $623$ & $623$ & $623$ & $623$ \\ 
$\nu_{16}^{*}$ & $1246$ & $1246$ & $1246$ & $1246$ & $1246$ & $1246$ & $1246$ & $624$ & $623$ & $623$ & $623$ & $623$ & $623$ & $623$ & $623$ & $623$ \\ 
$\nu_{17}^{*}$ & $1247$ & $1246$ & $1246$ & $1246$ & $1246$ & $1246$ & $1246$ & $1246$ & $623$ & $623$ & $623$ & $623$ & $623$ & $623$ & $623$ & $623$ \\ 
$\nu_{18}^{*}$ & {} & $1247$ & $1246$ & $1246$ & $1246$ & $1246$ & $1246$ & $1246$ & $624$ & $623$ & $623$ & $623$ & $623$ & $623$ & $623$ & $623$ \\ 
$\nu_{19}^{*}$ & {} & {} & $1246$ & $1246$ & $1246$ & $1246$ & $1246$ & $1246$ & $1246$ & $623$ & $623$ & $623$ & $623$ & $623$ & $623$ & $623$ \\ 
$\nu_{20}^{*}$ & {} & {} & {} & $1246$ & $1246$ & $1246$ & $1246$ & $1246$ & $1246$ & $624$ & $623$ & $623$ & $623$ & $623$ & $623$ & $623$ \\ 
$\nu_{21}^{*}$ & {} & {} & {} & {} & $1246$ & $1246$ & $1246$ & $1246$ & $1246$ & $1246$ & $623$ & $623$ & $623$ & $623$ & $623$ & $623$ \\ 
$\nu_{22}^{*}$ & {} & {} & {} & {} & {} & $1246$ & $1246$ & $1246$ & $1246$ & $1246$ & $624$ & $623$ & $623$ & $623$ & $623$ & $623$ \\ 
$\nu_{23}^{*}$ & {} & {} & {} & {} & {} & {} & $1246$ & $1246$ & $1246$ & $1246$ & $1246$ & $623$ & $623$ & $623$ & $623$ & $623$ \\ 
$\nu_{24}^{*}$ & {} & {} & {} & {} & {} & {} & {} & $1246$ & $1246$ & $1246$ & $1246$ & $624$ & $623$ & $623$ & $623$ & $623$ \\ 
$\nu_{25}^{*}$ & {} & {} & {} & {} & {} & {} & {} & {} & $1246$ & $1246$ & $1246$ & $1246$ & $623$ & $623$ & $623$ & $623$ \\ 
$\nu_{26}^{*}$ & {} & {} & {} & {} & {} & {} & {} & {} & {} & $1246$ & $1246$ & $1246$ & $624$ & $623$ & $623$ & $623$ \\ 
$\nu_{27}^{*}$ & {} & {} & {} & {} & {} & {} & {} & {} & {} & {} & $1246$ & $1246$ & $1246$ & $623$ & $623$ & $623$ \\ 
$\nu_{28}^{*}$ & {} & {} & {} & {} & {} & {} & {} & {} & {} & {} & {} & $1246$ & $1246$ & $624$ & $623$ & $623$ \\ 
$\nu_{29}^{*}$ & {} & {} & {} & {} & {} & {} & {} & {} & {} & {} & {} & {} & $1246$ & $1246$ & $623$ & $623$ \\ 
$\nu_{30}^{*}$ & {} & {} & {} & {} & {} & {} & {} & {} & {} & {} & {} & {} & {} & $1246$ & $624$ & $623$ \\ 
$\nu_{31}^{*}$ & {} & {} & {} & {} & {} & {} & {} & {} & {} & {} & {} & {} & {} & {} & $1246$ & $623$ \\ 
$\nu_{32}^{*}$ & {} & {} & {} & {} & {} & {} & {} & {} & {} & {} & {} & {} & {} & {} & {} & $624$ \\ \hline
$d(v)$ & $549$ & $484$ & $426$ & $373$ & $326$ & $283$ & $243$ & $207$ & $174$ & $143$ & $115$ & $89$ & $64$ & $41$ & $20$ & $0$ \\ \hline
$N_v$ &  $7$ & $6$ & $6$ & $6$ & $6$ & $6$ & $6$ & $6$ & $6$ & $6$ & $6$ & $6$ & $6$ & $6$ & $5$ & $5$ \\ \hline
\end{tabular}
}
\label{table:successive_minima2}
\end{table}

To conduct statistical tests in the next section, 
we introduce the minimal $\mathbb{F}_2$-linear relations with 21-bit and 12-bit accuracy, for example. 
First, we analyze the minimal $\mathbb{F}_2$-linear relations with $21$-bit accuracy. 
By checking all the nonzero shortest vectors in $\mathcal{L}_{21}^{*}$,
we obtain the following low-weight $\mathbb{F}_2$-linear relations:
the six-term linear relation 
\begin{eqnarray*}
y_{i, 1} + y_{i, 16} + y_{i + 396, 2} + y_{i + 396, 17} + y_{i + 623, 2}  + y_{i +623, 17}  =  0,
\end{eqnarray*}
the seven-term linear relations
\begin{eqnarray*}
& & y_{i, 7} + y_{i, 14} + y_{i, 15} + y_{i + 396, 8} + + y_{i + 396, 16} + y_{i + 623, 8} + y_{i + 623, 16} =  0,\\  
& & y_{i, 3} + y_{i + 396, 1} + +  y_{i + 396, 4} + y_{i + 396, 19} + y_{i + 623, 1} + y_{i + 623, 4}  + y_{i + 623, 19} =  0,\\
& & y_{i, 2} + y_{i, 9} + y_{i, 10} + y_{i, 17} + y_{i, 20} + y_{i + 396, 11} + y_{i + 623, 11} = 0,
\end{eqnarray*}
and so on. In particular, all of the minimal $\mathbb{F}_2$-linear relations 
concentrate on the three non-successive output values 
$\{ \mathbf{y}_{i}, \mathbf{y}_{i+396}, \mathbf{y}_{i+623} \}$. 

In addition to the above, we analyze the minimal $\mathbb{F}_2$-linear relations with $12$-bit accuracy.
In this case, we have three minimal $\mathbb{F}_2$-linear relations, i.e., the five-term linear relation
\begin{eqnarray} \label{eqn:5-term}
y_{i, 2} + y_{i + 792, 4} + y_{i + 792, 11} + y_{i + 1246, 4} + y_{i + 1246, 11}  = 0, 
\end{eqnarray}
and the 18-term and the 19-term linear relations. 
They appear only on the five non-successive output values $\{ \mathbf{y}_{i}, \mathbf{y}_{i + 396}, \mathbf{y}_{i + 623}, \mathbf{y}_{i + 792}, \mathbf{y}_{i + 1246} \}$.

Consequently, we aim to determine whether there are any observable
deviations for such non-successive output values. 

\begin{remark} \label{remark:tempering}
Niederreiter \cite{MR1236747, MR1334623} proposed the {\it multiple-recursive matrix method} 
as a general class of pseudorandom number generators. In this framework,
we can describe Mersenne Twisters in (\ref{eqn:MT transition matrix})--(\ref{eqn:MT output transformation2}) 
by the following matrix linear recurrence: 
\begin{eqnarray*} \label{eqn:MRMM}
\mathbf{y}_{i} = \mathbf{y}_{i + n_2 - n_1} + \mathbf{T} \tilde{\mathbf{A}}
\left(
\begin{array}{cc}
\mathbf{0} & \mathbf{0} \\
\mathbf{0} & \mathbf{I}_{r}\\
\end{array} \right) \mathbf{T}^{-1} \mathbf{y}_{i + 1 - n_1} 
+ \mathbf{T} \tilde{\mathbf{A}}
\left(
\begin{array}{cc}
\mathbf{I}_{w-r} & \mathbf{0}\\
\mathbf{0} & \mathbf{0} \\
\end{array} \right) \mathbf{T}^{-1} \mathbf{y}_{i - n_1}.
\end{eqnarray*}
From a comparison of the lower $r$ coordinates, 
it is easy to see that there exist $\mathbb{F}_2$-linear relations among 
$\{ \mathbf{y}_i, \mathbf{y}_{i + n_2 - n_1}, \mathbf{y}_{i +1 - n_1} \}$. 
The Couture--L'Ecuyer dual lattice method gives explicit $\mathbb{F}_2$-linear relations without direct matrix computations, 
and it is applicable not only for Mersenne Twisters defined by (\ref{eqn:MT transition matrix})--(\ref{eqn:MT output transformation2}) 
but also for general $\mathbb{F}_2$-linear generators \textcolor{blue}{(\ref{eqn:transition})--(\ref{eqn:real})}. 
\end{remark}

\section{Birthday spacings tests for non-successive output values} \label{sec:birthday spacings test}

In this section, we report statistical tests for non-successive output values of MT19937. 
\textcolor{blue}{We conduct the birthday spacings test, which was proposed by Marsaglia \cite{Marsaglia1985}, 
further studied in \cite{Knuth:1997:ACP:270146, MR1823110}, 
and implemented in the TestU01 package \cite{MR2404400}.
We consider the techniques that are similar to \cite{LatticeDeng, L'Ecuyer:2004:DRN:961292.961302} 
for some multiple recursive generators (e.g., \cite{Deng:2000:RNG}).}

Following the notations of \cite{MR1823110, MR2723077},
we introduce the testing procedure. 
We fix two positive integers, $n$ and $t$, and  
generate $n$ ``independent'' points $\mathbf{u}_0, \ldots, \mathbf{u}_{n-1}$ 
in the $t$-dimensional hypercube $[0, 1)^t$. 
For the hypercube, 
we partition it into $d^t$ cubic boxes of equal size by dividing $[0,1)$ into $d$ equal segments. 
These boxes are numbered from $0$ to $d^t-1$ in lexicographic order. 
Let $I_1 \leq I_2 \leq \cdots \leq I_n$ be the numbers of the boxes where 
these points have fallen, sorted by increasing order. 
Define the spacings $S_j := I_{j+1} - I_j$, for $j = 1, \ldots, n - 1$. 
Let $Y$ be the total number of collisions of these spacings, i.e., 
the number of values of $j \in \{1, \cdots, n-2\}$ such that $S_{(j + i)} = S_{(i)}$, 
where $S_{(1)}, \ldots, S_{(n-1)}$ are the spacings sorted by increasing order. 
We test the null hypothesis $\mathcal{H}_0$: the PRNG produces i.i.d.\ $U(0,1)$ random variables.
If $d^t$ is large and $\lambda = n^3/(4d^t)$ is not too large,
$Y$ is approximately a Poisson distribution with mean $\lambda$ under $\mathcal{H}_0$.
We generate independent $N$ replications of $Y$, add them, and  
compute the $p$-value by using the sum, which is approximately a Poisson 
distribution with mean $N \lambda$, 
under $\mathcal{H}_0$. 
If $d = 2^v$, note that the $t$-dimensional output with $v$-bit accuracy is tested.

To extract non-successive output values, 
let us consider the $t$-dimensional output vectors constructed as 
\textcolor{red}{$\mathbf{u}_i = (u_{(j_t+1)i + j_1}, \ldots, u_{(j_t+1)i + j_t})$} 
for $i = 0, \ldots, n-1$ with lags $I = \{ j_1, \ldots, j_t \}$. 


First, we conduct experiments with the parameter set 
$(N, n, d, t) = (5, 20000000, 2^{21}, 3)$, which is just No.~12 of Crush in TestU01.
Then, the three-dimensional output with $21$-bit accuracy is tested. 
\textcolor{blue}{We choose $I = \{ 0, 396, 623 \}$. } 
The second row in Table~\ref{mt_birthday} gives right $p$-values for five initial states, 
and all the $p$-values are $< 10^{-15}$.
Thus, MT19937 with $I = \{ 0, 396, 623 \}$ decisively 
fails the birthday spacings tests, \textcolor{red}{in accordance with} low-weight minimal 
$\mathbb{F}_2$-linear relations with $21$-bit accuracy in Section~\ref{sec:linear relation for MT}. 
It takes approximately eight minutes on an Intel Core i7-3770 3.90 GHz computer (with the gcc compiler with 
\textcolor{blue}{the} {\tt -O3} optimization flag on a Linux operating system) for each test.

Next, we \textcolor{blue}{also} conduct the birthday spacings tests five times
with the parameter set $(N, n, d, t) =  (5, 15000000, 2^{12}, 5)$. 
Thus, the five-dimensional output with $12$-bit accuracy is tested. 
The third row in Table~\ref{mt_birthday} shows small deviations 
for the points with $I = \{ 0, 396, 623, 792, 1246 \}$. 
It takes approximately 11 minutes for each test in the above environment. 
Furthermore, we focus on the five-term $\mathbb{F}_2$-linear relation (\ref{eqn:5-term}). 
\textcolor{blue}{The last row of Table~\ref{mt_birthday} shows similar deviations of the birthday spacings tests 
for $(N, n, d, t) = (5, 20000000, 2^{21}, 3)$ and $I = \{ 0, 792, 1246 \}$. 
In general, note that the discovery of such bad lag sets $I $ is not trivial when we use traditional criteria, such as $k(v)$ and $N_1$.}

\begin{table}[h]
\begin{center}
\caption{The $p$-values on the birthday spacings tests with selected lags $I$ for MT19937.}
{\small
\begin{tabular}{c||c|c|c|c|c} \hline
{} & {1st} & {2nd} & {3rd} & {4th} & {5th} \\ \hline  \hline
{$I = \{ 0, 396, 623 \}$} & $1.7 \times 10^{-16}$ & $1.8 \times 10^{-18}$ & $3.1 \times 10^{-21}$ & $8.5 \times 10^{-17}$ & $1.4 \times 10^{-21}$ \\ \hline 
{$I = \{ 0, 396, 623, 792, 1246 \}$} & $4.8 \times 10^{-5}$ & $0.01$ & $1.5 \times 10^{-4}$ & $1.1 \times 10^{-4}$ & $8.5  \times 10^{-4}$ \\ \hline 
{$I = \{ 0, 792, 1246 \}$} & $2.0 \times 10^{-4}$ & $3.0 \times 10^{-7}$ & $3.9 \times 10^{-6}$ & $9.2 \times 10^{-5}$ & $4.5  \times 10^{-6}$ \\ \hline 
\end{tabular}
}
\label{mt_birthday}
\end{center}
\end{table}

\section{Some variants of Mersenne Twister generators} \label{sec:other generators}

We analyze other $\mathbb{F}_2$-linear generators 
whose periods are $2^{19937}-1$ (i.e., $p = 19937$ and $w =32$). 
First, we investigate the WELL generators \cite{PannetonLM06}, 
which are variants of Mersenne Twister and 
have almost optimal $k(v)$ and $N_1$. 
A key idea of the improvement is to construct a more complicated transition matrix 
$\mathbf{A}$ in (\ref{eqn:MT transition matrix}) by using linear recurrences with a double loop, 
instead of that with a single loop \textcolor{red}{in Remark~\ref{remark:tempering}} (see \cite{MR2743921} for details). 
Panneton et al. \cite{PannetonLM06} list the parameters of 
WELL generators WELL19937a, which has $\Delta =4$, 
and WELL19937c, which has $\Delta=0$ (i.e., maximally equidistributed) 
by adding the Matsumoto--Kurita tempering \cite{MK1994}. 
The author \cite {Harase2009} also introduced more simplified temperings 
required to attain the maximal equidistribution. 
(We discovered a typo in Table~4 of \cite{Harase2009}. 
The bitmask {\tt 4202000} should be corrected to {\tt 4202010}.)
All of the WELL generators have $N_1 = 8585$ and $N_v >9500$ ($2 \leq v \leq 32$), 
so that we have no low-weight minimal $\mathbb{F}_2$-linear relations \textcolor{blue}{in (\ref{eqn:minimal}). 
This means that all the minimal $\mathbb{F}_2$-linear relations lie on a large number of coordinates (i.e., $\geq \lceil N_v/v \rceil$). 
In other words, among the $((k(v)+1) \times v)$ bits 
that correspond to the image of the map of (\ref{eqn:code}) with $k = k(v)+1$, 
there do not exist $\mathbb{F}_2$-linear relations concentrating on low-dimensional projections. 
Thus, as far as we focus on the $((k(v)+1) \times v)$ bits (in the same setting as MT19937), 
there are no corresponding bad lag sets $I = \{ j_1, \ldots, j_t \}$ for which the birthday spacings tests are rejected, because $t$ is too large. 
In this respect, the WELL generators are much superior to MT19937.}

As another improvement, the author \cite{Harase2009} constructed a maximally equidistributed Mersenne Twister 
MEMT19937 by replacing the tempering in \textcolor{red}{(\ref{eqn:MT output transformation1}) and (\ref{eqn:MT output transformation2})} 
with a more complicated output transformation $\mathbf{B}$, 
which consists of a linear combination of some part of the state vector. 
However, the author recently noted the following drawbacks: 
(i) MEMT19937 is sometimes slower than WEL19937a on some recent platforms; 
(ii) MEMT19937 has a small value $N_{32} = 26$. 
Again, we search for a better parameter set by using \textcolor{blue}{\cite{Shin2011141}}.   
By trial-and-error, as well as in \cite{Harase2009}, we obtain a simple linear transformation:
\textcolor{red}{
\begin{eqnarray*}
\mathbf{z} & \leftarrow & \mathbf{m}_i,\\
\mathbf{z} & \leftarrow & \mathbf{z} \ \oplus \ (\mathbf{m}_{i - 473} \ \& \ {\tt b219beab}),\\
\mathbf{z} & \leftarrow & \mathbf{z} \ \oplus \ (\mathbf{z} \ll 8),\\
\mathbf{z} & \leftarrow & \mathbf{z} \ \oplus \ (\mathbf{z} \ll 14),\\
\mathbf{y}_i & \leftarrow & \mathbf{z} \ \oplus \ (\mathbf{m}_{i - 588} \ \& \ {\tt 56bde52a}),
\end{eqnarray*}
}
where \textcolor{red}{the state vector $\mathbf{x}_i$ is decomposed into 
$\mathbf{x}_i =: (\mathbf{m}_i, \mathbf{m}_{i-1}, \ldots, \mathbf{m}_{i-n_1+2}, {\rm trunc}_{w-r}(\mathbf{m}_{i-n_1+1}))$,}
$\oplus$ denotes bitwise exclusive-or, $\&$ bitwise AND, $(\mathbf{z} \ll s_1)$ the $s_1$ bit left-shift, 
$(\mathbf{z} \gg s_2)$ the $s_2$ bit right-shift, and \textcolor{red}{${\tt b219beab}$ and ${\tt 56bde52a}$} are hexadecimal notations. 
We replace the tempering \textcolor{red}{(\ref{eqn:MT output transformation1}) and (\ref{eqn:MT output transformation2})} 
with the above, and 
we then obtain a maximally equidistributed generator 
\textcolor{red}{that has fewer operations than MEMT19937 and has almost the same number of operations as MT19937}. 
We name this MEMT19937-II. 
MEMT19937-II has $N_1 = 135$ and $N_v >9000$ ($2 \leq v \leq 32$), 
namely, $N_v$'s significantly increase. 
This generator passes the Big Crush suits in the TestU01 statistical test library, 
with the exception of two linear complexity tests (the test number 80 and 81), 
and these rejections are common among $\mathbb{F}_2$-linear generators, 
such as the Mersenne Twister and WELL generators (see \cite{MR2404400}). 
\textcolor{red}{Regarding linear complexity tests, Deng, Lu, and Chen \cite{MR2660671} recently
proposed the combined Mersenne Twister CMT19937, 
which passes the battery of tests including linear complexity tests in TestU01, 
by adding the outputs as real numbers (modulo $1$), as described in \cite{Deng2008401}.}
Now, MEMT19937-II is available at the author's homepage \textcolor{blue}{\url{http://www3.ocn.ne.jp/~harase/megenerators2.html}}.

Here, we measure the speed to generate $10^{\textcolor{blue}{9}}$ 32-bit unsigned integers on 
two different 64-bit CPUs: Intel Core i7-3770 3.90 GHz and AMD Phenom II X6 1045T 2.70 GHz. 
We use the gcc compiler with \textcolor{blue}{the} {\tt -O3} optimization flag on Linux operating system\textcolor{blue}{s}. 
In comparison, we also conduct experiments with MT19937ar, Shawn Cokus' other implementation MT19937ar-cok 
(both are obtained from \cite{MT_home}), \textcolor{red}{WELL19937a, and MEMT19937}. 
Table~\ref{benchmark} gives a summary of the CPU time (in seconds) and 
\textcolor{blue}{$\Delta$}.
MT19937ar-cok is the fastest, 
but the maximally equidistributed generator MEMT19937-II is comparable to \textcolor{red}{or faster than} MT19937ar on the two platforms. 

\begin{table}[h]
\begin{center}
\caption{CPU time (sec) taken to generate $10^{\textcolor{blue}{9}}$ pseudorandom numbers and total dimension defects $\Delta$.}
\begin{tabular}{c||c|c|c} \hline
{} & {Intel Core i7} & {AMD Phenom II} & $\Delta$ \\ \hline \hline
{MT19937ar-cok} & $3.126$ & $4.199$ & $6750$ \\ \hline
{MEMT19937-II} & \textcolor{red}{$4.093$} & \textcolor{red}{$5.841$} & $0$ \\ \hline 
{MT19937ar} & $4.771$ & $6.106$ & $6750$ \\ \hline
{WELL19937a} & $4.953$ & $6.678$ & $4$ \\ \hline
\textcolor{red}{MEMT19937} & \textcolor{red}{$5.111$} & \textcolor{red}{$7.812$} & \textcolor{red}{$0$} \\ \hline
\end{tabular}
\label{benchmark}
\end{center}
\end{table}

Finally, we conduct \textcolor{blue}{the} birthday spacings tests for non-successive output values of MEMT19937-II.  
Table~\ref{memt_birthday} gives a summary of the birthday spacings tests 
with the same parameter sets 
for three- and five-dimensional non-successive outputs in Section~\ref{sec:birthday spacings test}. 
A simple improvement of $\mathbf{B}$ also increases $N_v$'s, 
\textcolor{blue}{and the bad lag sets $I$ disappear, so that the} deviations of tests are avoided.

\textcolor{red}{
\begin{table}[h]
\begin{center}
\caption{The $p$-values of the birthday spacings tests five times for MEMT19937-II.}
\begin{tabular}{c||c|c|c|c|c} \hline
{} & {1st} & {2nd} & {3rd} & {4th} & {5th} \\ \hline  \hline
{$I = \{ 0, 396, 623\}$} & $0.11$ & $0.32$ & $0.87$ & $0.63$ & $0.03$ \\ \hline 
{$I = \{ 0, 396, 623, 792, 1246\}$} & $0.87$ & $0.80$ & $0.28$ & $0.33$ & $0.10$\\ \hline 
{$I = \{ 0, 792, 1246 \}$} & $0.76$ & $0.63$ & $0.85$ & $0.78$ & $0.67$ \\ \hline 
\end{tabular}
\label{memt_birthday}
\end{center}
\end{table}
}
\section{Conclusions} \label{sec:conclusions}
\textcolor{red}{To assess $\mathbb{F}_2$-linear pseudorandom number generators, w}e have discussed the relationship between $\mathbb{F}_2$-linear relations and the Couture--L'Ecuyer dual lattices, 
and have proposed the new figure of merit $N_v$ based on 
the minimum weight of $\mathbb{F}$-linear relations for most significant $v$ bits 
in $(k(v)+1)$-dimensional output vectors. 
We presented an algorithm for computing $N_v$, and applied it to the Mersenne Twister MT19937.
The experimental results showed that MT19937 has low-weight $\mathbb{F}_2$-linear relations, 
and is rejected for birthday spacings tests with specific lags, 
and the reason appears to be the existence of such $\mathbb{F}_2$-linear relations. 
To avoid such phenomena, some improvements of Mersenne Twister generators were also discussed. 

The above result of MT19937 will not affect most 
Monte Carlo simulations because real simulations are 
not synchronized with bad lag sets $I$. 
However, in general, when strange phenomena occur in simulations, 
and when these are due to the regularity of pseudorandom number generators, 
it is significantly difficult for experimenters to determine the reason 
for occurrence of the strange phenomena. 
Thus, when designing pseudorandom number generators, it is important to 
perform assessments in as many situations as possible beforehand.
In this respect, the Couture--L'Ecuyer lattices are powerful tools 
not only for computing $k(v)$ but also for detecting 
hidden structural defects of $\mathbb{F}_2$-linear generators. 

\subsection*{Acknowledgments}
\textcolor{red}{The author wishes to express his gratitude to Professor Makoto Matsumoto at Hiroshima University and 
Professor Syoiti Ninomiya at Tokyo Institute of Technology for continuous encouragement. 
The author is also grateful to the anonymous referees for many useful comments.} 
This work was partially supported by JSPS Research Fellowships for Young Scientists, 
JSPS Grant-In-Aid \#21654017,
 \#23244002, and Global COE Program ``The Research and Training Center for New Development in Mathematics'' from MEXT, Japan.




\bibliographystyle{model1b-num-names.bst}
\bibliography{latticebib}







\end{document}